\documentclass{amsart}
\usepackage{fullpage,amsfonts,amsmath,amscd,amssymb,graphicx}
\usepackage{epsfig,graphics,color}

\newcommand{\bK}{{\mathbb K}}
\newcommand{\bN}{{\mathbb N}}

\newcommand{\bR}{{\mathbb R}}

\newcommand{\bC}{{\mathbb C}}
\newcommand{\bF}{{\mathbb F}}
\newcommand{\bQ}{{\mathbb Q}}

\newcommand{\bP}{{\mathbb P}}

\newcommand{\bZ}{{\mathbb Z}}

\newcommand{\cC}{{\mathcal C}}

\newcommand{\cD}{{\mathcal D}}
\newcommand{\cO}{{\mathcal O}}

\newcommand{\cF}{{\mathcal F}}
\newcommand{\cG}{{\mathcal G}}
\newcommand{\cL}{{\mathcal L}}
\newcommand{\cP}{{\mathcal P}}
\newcommand{\cT}{{\mathcal T}}

\newcommand{\fg}{{\mathfrak g}}

\newcommand{\fh}{{\mathfrak h}}

\newcommand{\1}{{\mathbb I}}
\newcommand{\ba}{{\bf a}}
\newcommand{\bc}{{\bf c}}
\newcommand{\bl}{{\bf l}}
\newcommand{\br}{{\bf r}}
\newcommand{\shh}{{\bf s}{\bf h}}
\newcommand{\sh}{\shh_{n,m}}
\newcommand{\wh}{\widetilde{\shh}_{n,m}}

\renewcommand{\sl}{{\mathfrak s}{\mathfrak l}}
\renewcommand{\sp}{{\mathfrak s}{\mathfrak p}}
\newcommand{\SV}{{{\mathrm S}{\mathrm V}}}
\newcommand{\RW}{{{\mathrm R}{\mathrm W}}}
\newcommand{\SVX}{\SV(X)}
\newcommand{\TX}{\cT_X^-}
\newcommand{\Ext}{{{\mathrm E}{\mathrm x}{\mathrm t}}}
\newcommand{\Endd}{{{\mathrm e}{\mathrm n}{\mathrm d}}}
\newcommand{\End}{{{\mathrm e}{\mathrm n}{\mathrm d}_\cC}}
\newcommand{\moC}{{{\mathrm m}{\mathrm o}{\mathrm d}_{\cC}}}

\newcommand{\lb}{\left(\!\!}
\newcommand{\rb}{\!\!\right)}

\newtheorem{theorem}{Theorem}
\newtheorem{prop}[theorem]{Proposition}

\newtheorem{que}{Question}

\newenvironment{mylist}{\begin{list}{}{
\setlength{\itemsep}{0mm}
\setlength{\parskip}{0mm}
\setlength{\topsep}{1mm}
\setlength{\parsep}{0mm}
\setlength{\itemsep}{0mm}
\setlength{\labelwidth}{6mm}
\setlength{\labelsep}{3mm}
\setlength{\itemindent}{0mm}
\setlength{\leftmargin}{9mm}
\setlength{\listparindent}{6mm}
}}{\end{list}}

\begin{document}

\title{Lie algebras in symmetric monoidal categories}
\author{Dmitriy Rumynin}
\address{Department of Mathematics, University of Warwick, Coventry, CV4 7AL, UK}
\email{D.Rumynin@warwick.ac.uk}
\thanks{The author was partially supported by the Max Planck Institute for Mathematics, Bonn and Dinastiya Foundation.}
\date{May 25, 2012}
\subjclass{18D10, 53C26, 17B20}

\begin{abstract}
We study algebras defined by identities in symmetric monoidal categories.
Our focus is on Lie algebras. Besides usual Lie algebras, there are examples
appearing in the study of knot invariants and Rozansky-Witten invariants.
Our main result is a proof of Westbury's conjecture for K3-surface:
there exists a Lie algebra homomorphism from Vogel's universal
simple Lie algebra to the Lie algebra describing the Rozansky-Witten invariants
of a K3-surface.
Most of the paper involves setting up a proper language to discuss
the problem and we formulate nine open questions as we proceed.
\end{abstract}

\maketitle

In 1996 Deligne made a conjecture that exceptional Lie algebras formed a series \cite{Deli}.
He proposed a possible explanation: he expected an object specializing to all exceptional
Lie algebras. Influenced by this conjecture, Vogel proposed such an object in 1999 \cite{Vog2}.
Besides exceptional Lie algebras, it specializes to all simple (complex finite dimensional)
Lie algebras as well as some Lie superalgebras.
The object was named Vogel's universal simple Lie algebra $\fg_V$.
With Vogel's paper still unpublished the object remains a mystery. 
It is a subject of several recent studies \cite{LaMa, MkSV, West}.

The aim of this paper is to describe another specialization of $\fg_V$.
In a private conversation Westbury has asked the author whether $\fg_V$
specializes to $\fg_X$, a certain Lie algebra associated by Kapranov \cite{Kapr}
to an irreducible holomorphic symplectic manifold $X$ in the study of Rozansky-Witten
invariants \cite{RW}. By {\em Westbury's Conjecture} we understand existence of such a specialization.
We give an affirmative answer to Westbury's Conjecture in the case of a K3-surface. It remains
open in the general case.

Most of the paper is devoted to setting up a machinery 
just to ask the question rigorously.
There are several levels of generality in which
the setup can be laid out.
As opposed to a more general operadic language of Hinich and Vaintrob \cite{HiVa},
we choose a more elementary language in the spirit of Sawon \cite{Sawo}
to benefit a wider audience.

Here is a detailed description of the paper.
Section 1 contains all necessary definitions and facts about the categories.
In Section 2 we discuss algebras in the categories. We speculate on possible utility of
such algebras in the study of identities. 
In Section 3 we study the Lie algebras. 
We formulate three
different notions of simplicity. We introduce Vogel's universal Lie algebra.
In Section 4 we investigate the Lie algebras
that appear in the study of the Rozansky-Witten invariants
and give an affirmative answer to Westbury's conjecture in
the case of a K3-surface.

The author expresses especial gratitude to B. Westbury for formulating the conjecture and his interest in the work
and to A. Kuznetsov for explaining the stability of the tangent bundle on a K3-surface.
The author would like to thank 
A. Baranov, V. Kac, D. Panyushev, A. Rosly and J. Sawon 
for valuable discussions and useful information.


\section{Categories}
\subsection{Tensor categories}
By {\em a tensor category over a commutative ring} $\bK$ we understand a $\bK$-linear\footnote{The home-sets are $\bK$-modules and the compositions
are $\bK$-bilinear, for instance, if $\bK=\bZ$ this means that the category is preadditive.}
symmetric monoidal category 
$$
\cC = (\cC, \otimes, \ba, \1, \bl, \br, \bc)
$$
where 
$\otimes: \cC\times\cC\rightarrow\cC$ is the tensor product bifunctor, 
$\ba_{A,B,C}: (A\otimes B)\otimes C\rightarrow A\otimes (B\otimes C)$ is the associativity transformation, 
$\bc_{A,B}: A\otimes B\rightarrow B\otimes A$ is the symmetric braiding, 
$\1\in\cC$ is the monoidal identity, 
$\bl_{A}: \1\otimes A\rightarrow A$ and $\br_{A}: A\otimes \1\rightarrow A$ are left and right unity transformations.
We require the  structures to agree: 
the tensor product must be $\bK$-bilinear on morphisms
and the natural morphism $\bK \rightarrow \End (\1) =\hom_\cC (\1, \1)$
must be a ring isomorphism. 
We attempt to follow notations and terminology of Joyal and Street \cite{JoSt},
although their ``tensor category'' is the same as a monoidal category,
while our tensor categories are $\bK$-linear and symmetric.

Observe that $\cC$ is not required to be abelian or even additive. 
Additivity can be easily fixed by completing the category with respect to finite direct sums 
and the zero object
but not being abelian is 
an important distinction with other ``tensor categories'' that appear in the literature.
Our tensor categories are close to those of Hinich and Vaintrob \cite{HiVa}
with an exception of our additional assumption
that 
$\bK \rightarrow \End (\1)$
is an isomorphism.
This assumption is not very restrictive.
\begin{prop}
Suppose $\cC = (\cC, \otimes, \ba, \1, \bl, \br, \bc)$
is a symmetric monoidal additive category. Then the ring
$\bK_\cC = \End (\1)$ is commutative and
the category $\cC$ is a tensor category of $\bK_\cC$.
\end{prop}
\begin{proof}
The ring $\bK_\cC$ is commutative by the Eckmann-Hilton argument \cite[Prop 6.2]{KeLa}. 
Each $\hom_\cC (X, Y)$ is a $\bK_\cC$-$\bK_\cC$-bimodule via the left action
$$
\End (\1) \times \hom_\cC (X, Y) \rightarrow \hom_C (\1\otimes X, \1 \otimes Y) \cong \hom_\cC (X, Y)
$$
and the right action
$$
\hom_\cC (X, Y) \times \End (\1)  \rightarrow 
\hom_C (X\otimes\1 , Y \otimes \1) \cong \hom_\cC (X, Y).
$$
These actions coincide: given $\alpha\in \End (\1)$, $f\in \hom_\cC (X, Y)$,
both $\alpha f$ and $f\alpha $ can be read off the diagram
$$
\begin{CD} 
X  @> {\bl_X^{-1}} >> {\1\otimes X} @> {\alpha\otimes f} >> {\1\otimes Y}   @> {\bl_Y} >> Y\\ 
@VV {=} V  @V{\bc_{\1,X}} VV @VV {\bc_{\1,Y}} V @VV {=} V \\ 
X  @> {\br_X^{-1}} >> {X\otimes \1} @>>{f\otimes \alpha}> {Y\otimes\1} @> {\br_Y} >> Y \\ 
\end{CD} 
$$
as compositions from the top left corner to the bottom right corner.
The diagram is commutative: the outside squares are commutative 
by the standard fact \cite[Prop 2.1]{JoSt}
while the middle square is commutative because
the commutativity constraint $\bc$ is a natural transformation.
Hence, $\alpha f=f \alpha$.

Finally, $\bK_\cC$-bilinearity of compositions is evident.
\end{proof}

Tensor categories are ubiquitous in modern mathematics, so we spare the reader of an example now but
introduce some as we go along.

%
%

\subsection{Tensor powers in a tensor category}
Given an object $A$ in a tensor category over $\bK$, 
we define its iterated tensor products by
$A^{\otimes 0} := \1$
and
$A^{\otimes n}:= A^{\otimes (n-1)}\otimes A$.
There is an action of the symmetric group $S_n$
on the object $A^{\otimes n}$, i.e. a semigroup homomorphism
$$
S_n \rightarrow \End (A^{\otimes n}, A^{\otimes n}), \ \
\sigma \mapsto \widetilde{\sigma}_A
,$$
defined as follows. 
Using  a chain of associativity constraints
$$\gamma_i : A^{\otimes n} \rightarrow A^{\otimes (i-1)}\otimes ((A\otimes A) \otimes A^{\otimes (n-i-1)}),$$
we define it on transpositions by
$$
\widetilde{(i,i+1)}_A :=  \gamma_i^{-1}\circ (I\otimes (\bc_{A,A})\otimes I)) \circ \gamma_i
$$
and extend to the whole group: existence of such an extension follows from the axioms of symmetric monoidal category.
It gives a $\bK S_n$-module structure on
$\hom_\cC (A^{\otimes n}, B)$ 
for a pair of objects $A,B$ of $\cC$. 

\subsection{Extension of scalars}
Given a tensor category $\cC$ over $\bK$ and ring homomorphisms
$\bF\rightarrow\bK$ and
$\bF\rightarrow\bK^\prime$
of 
commutative rings, 
we can construct a new tensor category $\cC^\prime = \cC\otimes_{\bF}\bK^\prime$
over $\bK\otimes_{\bF}\bK^\prime$. It has the same objects as $\cC$ 
but new morphisms
$\hom_{\cC^\prime}(X,Y) := \hom_\cC (X,Y)\otimes_{\bF}\bK^\prime$.
The compositions are defined in the obvious way:
$(f\otimes \alpha) \circ (g \otimes \beta) = fg \otimes \alpha\beta$.
Similarly, all the natural transformations are extended to $\cC^\prime$.
All verifications are routine
and left to an interested reader.

\subsection{Tensor functors}
Let $\cC$ and $\cD$ be tensor categories over the same commutative ring $\bK$.
{\em A tensor functor } $F=(F_1,F_2,F_3): \cC \rightarrow \cD$
consists of a $\bK$-linear functor $F_1:\cC\rightarrow \cD$, a natural isomorphism
$F_{2,A,B}: FA \otimes FB \rightarrow F(A\otimes B)$ and an isomorphism $F_3 : \1_\cD \rightarrow F\1_\cC$.
They have to satisfy the standard axioms \cite{JoSt}.

\section{Algebras} 

\subsection{Algebras}
We can talk about algebras in $\cC$ over a $\bK$-linear operad or
even a $\bK$-linear prop\footnote{PROP is an abbreviation of ``product
  and permutation category'', introduced by MacLane.} \cite{HiVa}.
{\em A $\bK$-linear prop} is a tensor category $\cP$ over $\bK$ satisfying the following five conditions:
\begin{mylist}
\item[(1)] its objects form the monoid $(\bN,+)$, i.e. $n\otimes m = n+m$,
\item[(2)] it is strict, in this context it means that all $\ba_{n,m,k}=I_{n+m+k}$, $\bl_n=I_n$, and $\br_n=I_n$,
\item[(3)] for each object $n>0$ a semigroup embedding 
$S_n \hookrightarrow \hom_\cP (n,n)$, $\pi \mapsto \widetilde{\pi}$ is chosen,
\item[(4)] for all $\pi\in S_n$, $\sigma\in S_m$ their tensor product is another permutation $\tau\in S_{n+m}$:
$$
\widetilde{\pi}\otimes\widetilde{\sigma} = \widetilde{\tau}
\mbox{ where }
\tau (k) =  \left\{\!\!\begin{array}{lrr}\pi(k) & \mbox{ if } & k\leq n\\ \sigma(k-n)+n & \mbox{ if } & k> n  \end{array}\!\! \right.
,
$$
\item[(5)] the commutativity constraint is the shuffle $\sh\in S_{n+m}$:
$$
\bc_{n,m} = \wh
\mbox{ where }
\sh (k) =  \left\{\!\!\begin{array}{lrr}k+m & \mbox{ if } & k\leq n\\ k-n & \mbox{ if } & k> n  \end{array}\!\! \right. 
.
$$
\end{mylist}

Now {\em an algebra in } $\cC$ over a prop $\cP$ is a tensor functor $\cP\rightarrow \cC$.
This level of generality is not always useful as the prop needs to be constructed.
A more elementary approach
is more suitable for our purposes.
{\em A signature} $I=(I,h,t)$ is a set of operations 
$I$ 
with head and tail functions 
$h,t: I \rightarrow \bZ_{\geq 0}$, so that we think of $i\in I$
as an operation on $t(i)$ variables with $h(i)$ values.
{\em An algebra of signature}  $I$ 
is a triple 
$(\cC, A, (m_i, i\in I))$ where $\cC$ is a tensor category over a commutative ring 
$\bK$, $A$ is an object of $\cC$, and $m_i \in \hom_\cC (A^{\otimes t(i)}, A^{\otimes h(i)})$
is a family of morphisms. 
If $\bK$ or $\cC$ is fixed, we talk about $\bK$-algebras or $\cC$-algebras.
We specify a signature by listing the size of $I$ and pairs $(t(i),h(i))$, thinking
that $I=\{1,2, \ldots n\}$ and the $k$-th pair are the tail and head values of $m_k$.

Let us contemplate the connection between these two notions of algebra.
A prop $\cP$ admits a generating set of morphisms $I=\{i\,\mid\, i\in\hom_\cP (h(i),t(i))\}$.
An algebra $F:\cP\rightarrow \cC$ over a prop $\cP$ determines the algebra 
$(\cC, F_1(1), (F_1(i), i\in I))$ of the signature  $I=(I,h,t)$, determined by the generating set.
In the opposite direction, an algebra 
$(\cC, A, (m_i, i\in I))$ of the signature  $I=(I,h,t)$
may or may not determine an algebra over $\cP$.
In general, it determines only an algebra over a free prop $\cP (I)$.
The prop $\cP$ is a quotient of the free prop $\cP (I)$,
so to determine an algebra over $\cP$ 
the algebra must satisfy some axioms, coming from $\cP$.
Thus, the props provide an implicit method of requesting axioms on algebras, but
often it is more convenient to do explicitly.

In the partial case of  an algebra $A$ of signature $[1:(2,1)]$
the axioms are just polylinear identities.
Let 
$F=\bK\{x_1,...\}$ be the free nonassociative algebra in countably many variables,
$F_n$ the subset of polylinear elements of degree $n$, linear in fixed variables, say 
$x_1,\ldots x_n$.
There is a natural ``evaluation'' map of $\bK$-modules
$$ 
\psi_A : F_n \rightarrow  \hom_\cC (A^{\otimes n},A),
$$
defined recursively on monomials: 
$$ 
\psi (x_1) = I_A, \ 
\psi (x_1x_2) = m_1, \ 
\psi (v(x_{\sigma (1)}, \ldots x_{\sigma (k)} ) w(x_{\sigma (k+1)}, \ldots x_{\sigma (n)})) = 
m_1 \circ (\psi (v) \otimes \psi(w)) \circ \gamma \circ \widetilde{\sigma^{-1}}_A \ 
$$
where $v\in F_k, \; w\in F_{n-k}, \; \sigma \in S_n$
and
$\gamma: A^{\otimes n} \rightarrow A^{\otimes k} \otimes A^{\otimes (n-k)}$
is the composition of associativity constraints, and then extended to
linear combinations by
$\psi (\sum_i \alpha_i v_i) = \sum_i \alpha_i \psi (v_i)$.
Now we say that $A$ satisfies a polylinear identity $f$ if $\psi_A(f)=0$.


{\em A Lie algebra} is an algebra of signature $[1:(2,1)]$
satisfying the Jacobi and the anticommutativity identities:
$$
(x_1x_2)x_3+(x_2x_3)x_1+(x_3x_1)x_2 \ \mbox{ and } \ x_1x_2+x_2x_1,
$$
which can be written explicitly in the tensor notation as 
$$
m_1 \circ (m_1 \otimes I) \circ (\widetilde{1}+\widetilde{(2,3,1)}+\widetilde{(3,1,2)})=0
, \
m_1 \circ (\widetilde{1}+\widetilde{(1,2)})=0.
$$
Notice that if $\cC$ is the category of $\bK$-modules
and $\frac{1}{2} \in \bK$ then our definition is just the usual definition of a Lie algebra over $\bK$.
However, if $\frac{1}{2} \not\in \bK$, our anticommutativity is weak, i.e., $xy +yx =0$ but we cannot conclude that
$x^2=0$.

{\em An associative algebra} is an algebra of signature $[2:(2,1),(0,1)]$
satisfying the associativity 
$
(x_1x_2)x_3 
= 
x_1(x_2x_3)
$
and the left and right unities. The latter cannot be written in terms of $\psi_A$
so we recourse to the tensor notation formulating the axioms: 
$$
m_1 \circ (m_1 \otimes I) = m_1 \circ (I \otimes m_1) \circ \ba_{A,A,A}
, \
m_1 \circ (m_2 \otimes I) = \bl_A 
, \
m_1 \circ (I \otimes m_2) = \br_A .
$$


{\em A metric object} is an algebra of signature $[2:(0,2),(2,0)]$
such that both structures are symmetric:
$$
m_2 = m_2 \circ \bc_{A,A}, \ \ 
m_1 = \bc_{A,A} \circ m_1,
$$ and 
$A$ is a dual object of $A$ in the sense of monoidal categories, i.e.,
the identity
$$
I_A = \br_A \circ (I_A \otimes m_2) \circ \ba_{A,A,A} \circ (m_1\otimes I_A) \circ \bl_A^{-1}
$$
holds. Notice that $m_2$ is ``non-degenerate'' in this case.
If 
$(A^\ast, u\in \hom (A\otimes A^\ast, \1), c\in \hom (\1, A^\ast\otimes A))$
is another dual object, then
$$
A \xrightarrow{\bl_A^{-1}} \1 \otimes A
\xrightarrow{c\otimes I} (A^\ast \otimes A)\otimes A 
\xrightarrow{\ba_{A^\ast,A,A}}  A^\ast \otimes (A\otimes A) 
\xrightarrow{I\otimes m_2} A^\ast \otimes \1
\xrightarrow{\br_{A^\ast}} A^\ast
$$ 
is a canonical isomorphism.
Following Hinich and Vaintrob \cite{HiVa}, we are introducing Casimir and metric Lie algebras in the next section. 

\subsection{Representations}
Let $(\cC, \fg, m)$ be a Lie algebra.
{\em A representation} of $\fg$ is a pair $(M,\rho)$ where $M$ is an object in $\cC$,
$\rho\in\hom_\cC (\fg\otimes M, M)$ and the Jacobi identity
$$
\rho \circ (m\otimes I_M) = \rho \circ (I_\fg\otimes \rho) \circ \ba_{\fg,\fg,M} 
-   \rho \circ (I_\fg\otimes \rho) \circ \ba_{\fg,\fg,M} \circ (\bc_{\fg,\fg}\otimes I_M) 
\in\hom_{\cC}((\fg\otimes\fg)\otimes M, M)
$$
holds. Completely parallel to the case of Lie algebras in vector spaces,
the representations form a tensor category $\moC (\fg)$.
The morphisms $\hom_{\fg} (M,N)$ are those morphisms in $\hom_{\cC} (M,N)$
that commute with the action of $\fg$.
The tensor product is the tensor product in $\cC$ 
with the usual action:
$$
\rho_{M\otimes N} = 
(\rho_M \otimes I_N)\circ \ba_{\fg,M,N}^{-1} +  
(I_M \otimes \rho_N) \circ \ba_{M,\fg,N} \circ (\bc_{\fg,M}\otimes I_M)\circ \ba^{-1}_{\fg,M,N} 
\in\hom_{\cC}(\fg\otimes (M\otimes N), M\otimes N).
$$
The constraints are inherited from $\cC$. Let us summarise this in a proposition whose proof is left to an interested reader.
\begin{prop}
Let $\fg$ be a Lie $\cC$-algebra.
Then the category $\moC (\fg)$ is a tensor category over $\bK_\cC$.
Its identity object is $\1_{\moC (\fg)}=\1_\cC$ with zero action.
The multiplication in $\fg$ is a morphism in $\moC (\fg)$,
so $\fg$ is a Lie $\moC (\fg)$-algebra.
\end{prop}
One could ask what sort of tensor categories one gets by iteration: consider
representations of $\fg$ in $\moC (\fg)$, etc. The answer is rather uninteresting:
the reader can easily verify that 
subject to the existence of finite direct sums in $\cC$, 
the category of representations of $\fg$ in $\moC (\fg)$
is canonically equivalent to the category of representations of the semidirect product $\fg \ltimes \fg_a$ in $\cC$
(where $\fg_a$ is the adjoint representation of $\fg$ treated as a Lie
algebra with the zero Lie product).

Now {\em a Casimir Lie algebra}  
is a $\cC$-algebra of signature $[2:(2,1),(0,2)]$
such that $m_1$ defines a Lie algebra while $m_2$ is a symmetric 
morphism
in the category of $\fg$-modules.

Similarly {\em a metric Lie algebra} is a $\cC$-algebra $\fg$ of signature $[3:(2,1),(0,2),(2,0)]$
such that $(\fg, m_1)$ is a Lie $\cC$-algebra, 
and $(\fg, m_2, m_3)$ is a metric $\moC (\fg)$-object.

\subsection{Homomorphisms}
Naturally there are two notions of a homomorphism between algebras of the same signature. 
{\em A homomorphism from} $(\cC,A, m_i)$ {\em to} $(\cC,B, m^\prime_i)$
is a morphism $\varphi\in \hom_{\cC} (A,B)$ which commutes with all multiplications, i.e.
the diagram
$$
\begin{CD} 
{A^{\otimes h(i)}} @> {\varphi^{\otimes h(i)}} >> {B^{\otimes h(i)}} \\ 
@V{m_i} VV @VV {m^\prime_i} V \\ 
{A^{\otimes t(i)}} @>>{\varphi^{\otimes t(i)}}> {B^{\otimes t(i)}} \\ 
\end{CD} 
$$
is commutative
for all $i$. 

To avoid confusion we use the term specialization for the second notion.
{\em A specialization from}
$(\cC,A, m_i)$ {\em to} $(\cD,B, m^\prime_i)$
is a triple $\psi=(\psi_1, \psi_2, \psi_3)$ where $\psi_1 :\bK_\cC \rightarrow \bK_\cD$ 
is a ring homomorphism,
$\psi_2: \cC\otimes_{\bK_\cC} \bK_D \rightarrow \cD$ is a tensor functor, 
and $\psi_3$ is a homomorphism from $\psi_2 (A)$ to $B$.
A specialization such that $\psi_3$ is an isomorphism is called {\em
  isospecialization}.

Notice that we may consider the three types of homomorphisms (or specializations)
between metric Lie algebras: a Lie algebra homomorphism, 
a Casimir Lie algebra homomorphism, and 
a metric Lie algebra homomorphism.
The difference is in the preserved operations.

\subsection{Identities of algebras in categories}
Varieties of algebras is a topic of an active study in Algebra \cite{BaOl}.
Should one study varieties of algebras in categories?
In this subsection we make some easy observations 
and pose some questions.

At the first glance, at least if $\bK$ is a $\bQ$-algebra,
the varieties of algebras in categories are not much different
from the varieties of usual algebras. Indeed, an interested reader could
easily verify that the identities of a $\cC$-algebra 
is equal to the set of polylinear elements in some $T$-ideal (an ideal closed under substitutions)
of
the free algebra $\bK\{x_1,...\}$.
For instance, the following Engel-type theorem is an immediate consequence
of the same theorem for usual Lie algebras \cite[Th 6.4.1]{Kost}.
\begin{theorem}
Suppose $\bK$ is a commutative $\bQ$-algebra, $\cC$ is a tensor category over $\bK$, 
and $\fg$ is a Lie $\cC$-algebra
satisfying the linearised Engel identity. Then $\fg$ satisfies the linearised nilpotency.
\end{theorem}

%
%
%

Nevertheless, in our view, it is interesting 
to understand examples of $\cC$-algebras,
in particular, what identities they satisfy.
For instance, if $X$ is a rigid object in a tensor category $\cC$
then $E_X=X\otimes X^\ast$ is an associative $\cC$-algebra, so-called
{\em the internal endomorphism algebra}.
The identities of
$E_X$ in the free associative algebra $\bK <x_1,x_2 \ldots>$ 
form a $T$-ideal $I(E_X)\lhd \bK <x_1,x_2 \ldots>$.
It is interesting to study it \cite{BaOl}.
\begin{que}
Given a rigid object $X$ in $\cC$, what is the minimal degree
of an element of $I(E_X)$?
\end{que}

If $\cC$ is the category of vector  spaces, $X$ is an $n$-dimensional vector space then
$E_X$ is the algebra of $n\times n$ matrices,
whose minimal identity has degree $2n$ by the Amitsur-Levitzki theorem \cite{AmLe}.
If $\cC$ is the category of supervector  spaces
(this is the same as superbundles on a point, see Section~\ref{superbundle}), 
$X$ is an $(n,k)$-dimensional supervector space then
$E_X$ is the matrix superalgebra
whose minimal identity 
is conjectured to be of 
the degree $2(nk + n + k) - \min\{ n, k \}$ \cite{Samo}.

One reason the matrix superalgebras are of interest is that
in characteristic zero they generate all prime varieties 
of associative algebras \cite{Keme}. Is it conceivable that 
algebras $E_X$ generate all varieties?
\begin{que}
Give an example of a commutative $\bQ$-algebra $\bK$ and a $T$-ideal in $\bK <x_1,...>$
that is not an ideal of identities of $E_X$ for any rigid object $X$ in any tensor category $\cC$ over $\bK$.
\end{que}

We finish this section explaining the categorical analogue of the Grassmann envelope
used to study the varieties of superalgebras.
{\em The global sections functor} $\Gamma : \cC\rightarrow \mod(\bK_\cC)$ is defined by
$\Gamma (A) := \hom_\cC (\1, A)$.
If $A$ is an algebra in $\cC$ then
$\Gamma (A)$ is a $\bK_\cC$-algebra in 
$\bK_\cC$-modules.
The following proposition is evident. 
\begin{prop}
\label{decat_alg}
Suppose $\cC$ is a tensor category, 
$A$ is a $\cC$-algebra, $C$ is an associative commutative $\cC$-algebra.
Then $\Gamma (A\otimes C)$ is a $\mod(\bK_\cC)$-algebra of the same 
signature as $A$. Moreover, 
$\Gamma (A\otimes C)$ satisfies all identities of $A$.
\end{prop}

The Grassmann algebra $G_\infty := \bK <x_1,x_2 \ldots>/(x_ix_j+x_jx_i)$ in countably many variables 
is an associative commutative algebra in the category $\cC$ of supervector spaces.
If $A$ is a superalgebra then $\Gamma (A\otimes G_\infty)$
is the Grassmann envelope of $A$.
The identities of $\Gamma (A\otimes G_\infty)$
are precisely the identities of $A$ in $\cC$.
Can this be carried out in any category?
\begin{que}
Let $\cC$ be a cocomplete\footnote{closed under arbitrary direct sums.}
tensor category with split idempotents\footnote{every idempotent
  $e\in\hom_\cC (X,X)$
is the projection operator $X\xrightarrow{\cong} Y\oplus Z \rightarrow
Y \hookrightarrow X$.}.
Does there exist a commutative associative algebra $G$ in $\cC$
such that for any $\cC$-algebra $A$ all identities of  $\Gamma (A\otimes G)$ hold in $A$?
\end{que}

%

\section{Simple Lie algebras} 

\subsection{Quasisimple Lie algebras}
Following Vogel \cite{Vog1,Vog2},
a Lie algebra $(\cC,\fg,m_1)$ is
{\em quasisimple}
if it can be extended to a metric Lie algebra $(\cC,\fg,m_1,m_2,m_3)$
and 
the natural map $\bK_\cC \rightarrow \hom_{\fg} (\fg,\fg)$ is an isomorphism.
We discuss how it is related to
the usual simplicity in the next section.

\subsection{Simple Lie algebras}
If the $\bK$-tensor category $\cC$ is abelian, 
we can talk about subobjects, quotients, kernels, ideals, etc.
in the usual way. In particular, 
there is the usual notion of a simple Lie algebra.
However, the category being abelian is not required to introduce simplicity,
although the utility of this notion is unclear.
Neither is it clear whether this definition is useful for other classes of algebras.

Recall that a morphism $f\in\hom_\cC (X,Y)$ is {\em a monomorphism}
if it can be cancelled on the left, i.e., $fa=fb$ implies $a=b$
for all $a,b\in \hom_\cC (W,X)$.
A Lie algebra $\fg$ in $\cC$ is {\em simple} 
if $\fg\not\cong 0$ (notice that if a zero object exists, it is unique up to a canonical isomorphism)
and 
every 
algebra homomorphism $f:\fg\rightarrow\fh$
is a monomorphism or the zero.
Notice the fine difference 
with the usual simple Lie
algebras in vector spaces: the latter definition traditionally excludes the one-dimensional Lie algebra while our definition does not.

In general, there is no relation between simple and quasisimple Lie algebras.
Let
$\bK$ be a commutative ring, not a field, $1/2\in\bK$.
It admits a proper ideal $I$.
Then the Lie algebra $\sl_2(\bK)$ (in the category of $\bK$-modules)
is quasisimple: the trace form $m_2 (X\otimes Y):=\mbox{Tr}(XY)$
gives a metric structure with a coform
$$
m_3 (\alpha) := \alpha (
\lb \begin{array}{cc} 
0 & 1
\\ 
0 & 0
\end{array}\rb
\otimes 
\lb \begin{array}{cc} 
0 & 0
\\ 
1 & 0
\end{array}\rb
+
\lb \begin{array}{cc} 
0 & 0
\\ 
1 & 0
\end{array}\rb
\otimes 
\lb \begin{array}{cc} 
0 & 1
\\ 
0 & 0
\end{array}\rb
+\lb \begin{array}{cc} 
1 & 0
\\ 
0 & -1
\end{array}\rb
\otimes 
\lb \begin{array}{cc} 
\frac{1}{2} & 0
\\ 
0 & -\frac{1}{2}
\end{array}\rb
).$$
It is not simple because the quotient homomorphism $\sl_2 (\bK)\rightarrow\sl_2 (\bK/I)$
is neither zero, nor a monomorphism. 

Here is another example of a quasisimple non-simple 
Lie algebra.
Let $X$ be a set of simple finite dimensional complex Lie algebras,
$C=\bC X$ the group coalgebra.
A $C$-comodule is just an $X$-graded vector space $M=\oplus_{x\in X} M_x$.
We consider the tensor category $\cC$ of $C$-comodules 
with the standard tensor coproduct as a tensor product: $M\otimes N:= \oplus_{x\in X} (M_x \otimes_\bC N_x)$.
Now
$\fg = \oplus_{\fh\in X} \fh$
is a metric Lie $\cC$-algebra under
the direct sums of the products, the forms and the coforms 
$$
\bigoplus_{\fh\in X} m_\fh : \oplus (\fh\otimes\fh)\rightarrow \oplus \fh, \
\bigoplus_{\fh\in X} m_\fh : \oplus (\fh\otimes\fh)\rightarrow C, \
\bigoplus_{\fh\in X} m_\fh : C\rightarrow\oplus(\fh\otimes\fh)
.
$$
Observe that $\fg$ is quasisimple but, if $X$ contains at least two elements, not simple
since every $\fh$ is a proper ideal of $\fg$ resulting in the quotient homomorphism.

In the opposite direction,
$\sl_2(\bC)$ is a simple Lie algebra in the category of real
vector spaces but it is not quasisimple because
$\hom_{\fg} (\fg,\fg) = \bC \neq \bR$.
Nevertheless, sometimes there is a relation.
\begin{theorem} 
\label{simple_in_vs}
Let $\cC$ be the category of vector spaces over an algebraically closed
field $\bK$ of characteristic zero.
Then quasisimple Lie $\cC$-algebras are exactly the finite dimensional 
simple\footnote{Notice that this includes the one-dimensional Lie algebra according to our definition}
Lie algebras.
\end{theorem}
\begin{proof}
A finite dimensional simple Lie algebra $\fg$
is quasisimple because the Killing form is non-degenerate and
$\bK\rightarrow \Endd_\fg(\fg)$ 
is an isomorphism by the Schur lemma and algebraic closeness of $\bK$.

In the opposite direction, the space of symmetric invariant form $\fg\otimes\fg\rightarrow\bK$
is one-dimensional since it is a subspace of
$(\fg^*\otimes \fg^*)^\fg \cong \hom_\fg(\fg,\fg)=\bK$.
By the theorem of Bajo and Benayadi \cite{BaBe},
$\fg$ is simple.
\end{proof}

It would be interesting to find some other categories where 
an analogue of Theorem~\ref{simple_in_vs}
holds or where the quasisimple Lie algebras admit a meaningful classification.
Let us pose two precise questions.
\begin{que}
Let $G$ be an affine group scheme over 
a field $\bK$ of characteristic zero.
Let $\cC$ be the category of rational $G$-modules.
What are simple and quasisimple Lie algebras in $\cC$?
\end{que}

To get a feel of this question, 
let us consider an absolutely irreducible finite dimensional $G$-module $V$ as a Lie algebra
with the zero multiplication in $\cC$. According to our definition, $V$ is simple.
On the other hand, $V$ is quasisimple if and only if $V\cong V^\ast$.
\begin{que}
Let $\cC$ be the category of supervector spaces\footnote{It is
  equivalent to
the category of rational $C_2$-modules as a monoidal category, but not
as a tensor category.}
over an algebraically closed
field $\bK$ of characteristic zero.
Is a quasisimple Lie $\cC$-algebra a finite dimensional 
simple
Lie superalgebra?
\end{que}

Finite dimensional simple Lie superalgebras
are classified by Kac \cite{Kac}.
Simple Lie algebras and simple Lie superalgebras of types
$A(m, n)$, $n\neq m$, $B(m, n)$, $C(n)$, $D(m, n)$, $m-n\neq 1$, $F(4)$, and $G(3)$
have nondegenerate Killing forms
\cite[Th. 1]{Kac}.
Hence, they are quasisimple Lie $\cC$-algebras.
The remaining simple Lie superalgebras of types $A(n,n)$, $D(n + 1, n)$, $P(n)$, $Q(n)$, and
$D(2, 1, \alpha)$ 
as well as Cartan type superalgebras
have zero Killing forms \cite[Prop. 2.4.1]{Kac},
although there may be an invariant form distinct from the Killing form.
For instance, this happens in types $Q(n)$
and $D(2, 1, \alpha)$.
Her is a partial result towards this question.
\begin{prop}
\label{qs_super}
Let $\fg$ be a quasisimple Lie algebra
in the category of supervector spaces over a 
field $\bK$.  
Then 
\begin{mylist}
\item[(1)] $\fg$ is perfect,
\item[(2)] the centre of $\fg$ is zero,
\item[(3)] every minimal ideal of $\fg$ is abelian.
\end{mylist}
\end{prop}
\begin{proof}
Let $I$ be a minimal ideal. 
Its orthogonal complement $I^\perp$ is also an ideal.
Hence, the intersection $I\cap I^\perp$ is zero or $I$.
If $I\cap I^\perp=0$,
then $\fg = I \oplus I^\perp$
and we get nontrivial idempotents in $\bK=\Endd_\fg (\fg )$.
Hence $I\cap I^\perp=I$, so $I\subseteq I^\perp$. This means that 
$
([x,y],a)=(x,[y,a])=0
$
for all $x,y\in I$, $a\in \fg$.
As the form is nondegenerate, $[x,y]=0$ and $I$ is abelian.

Now
$[\fg,\fg]^\perp = \{ a\in\fg \mid \forall x,y\in \fg \ ([x,y],a)=(x,[y,a])=0\} =
\{ a\in\fg \mid \forall y\in \fg \ [y,a]=0\} = Z(\fg)$.
We conclude 
that $\fg$ is perfect and $Z(\fg)=0$.
Indeed, otherwise there exists a nonzero bilinear form $\alpha$ on $\fg/[\fg,\fg]$.
Extend it to a form $\widetilde{\alpha}$ on $\fg$: 
$\widetilde{\alpha} (x,y) := \alpha (x+[\fg,\fg],y+[\fg,\fg])$.
The form $\widetilde{\alpha}$ is invariant 
($\widetilde{\alpha} ([x,y],z)=0=\widetilde{\alpha} (x,[y,z])$),
nonzero and degenerate. 
This contradicts the one-dimensionality of $\Endd_\fg (\fg)\cong (\fg\otimes\fg)^{\ast\fg}$.
\end{proof}

Using Proposition~\ref{qs_super} and Cartan's criteria \cite[4.3 and 5.1]{Hum}, 
we can furnish an alternative proof of
Theorem~\ref{simple_in_vs}.
Consider $\fg$ as a superalgebra with zero odd part, hence
Proposition~\ref{qs_super} applies.
Since $\fg$ is not solvable, the Killing form is nonzero. So the Killing form
is a multiple of the non-degenerate form, and itself nondegenerate.
Thus, $\fg$ is semisimple. It is simple since $\Endd_\fg (\fg)=\bK$.
This proof fails for the superalgebras because of the lack of the Cartan criteria.

\subsection{Universal metric Lie algebra}
\label{UmLa}
The universal metric Lie algebra 
appears in the study of Vassiliev invariants 
\cite{NiWi,Vog2}.
The relevant symmetric monoidal category $\cC$ is a prop.
The hom-set $\hom_\cC(m,n)$ is the $\bQ$-vector space of 
Jacobi diagrams\footnote{also known as Chinese characters}
\cite{NiWi}.
Recall that an $(m,n)$-Jacobi diagram is a compact curve $X$ such that
\begin{mylist}
\item[(1)] the boundary of $X$ is the set $\{1,2,\ldots m, 1^\prime, 2^\prime \ldots n^\prime\}$,
\item[(2)] $X$ has finitely many trivalent singular points, i.e. points $x$ with a neighbourhood $U$ such
that $U\setminus \{x\}$ is a union of three lines,
\item[(3)] for each trivalent singular point $x$ a cyclic ordering on the components of $U\setminus \{x\}$ is fixed.
\end{mylist}

Here is an example of the Jacobi diagram in $\hom_\cC(3,2)$:

\begin{picture}(50,50)(-140,0)
%
\includegraphics[scale=0.4]{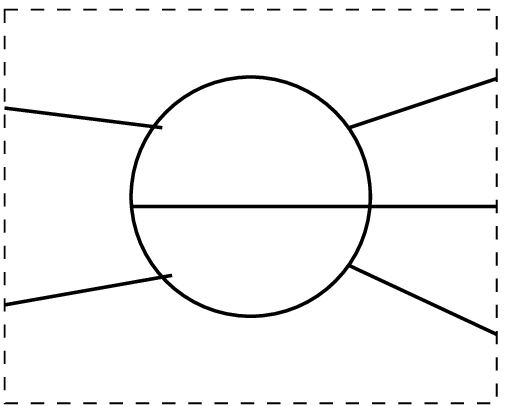}
\put(5,40){\makebox(0,0){$1$}}
\put(5,23){\makebox(0,0){$2$}}
\put(15,23){\makebox(0,0){.}}
\put(5,6){\makebox(0,0){$3$}}
\put(-63,34){\makebox(0,0){$1^\prime$}}
\put(-63,11){\makebox(0,0){$2^\prime$}}
\put(-18,12){\makebox(0,0){$\star$}}
\end{picture}

%
%
%
\noindent The external dashed borders have no significance. 
This diagram has five trivalent singular points,
four of them have the standard counterclockwise ordering, while the remaining vertex
marked with $\star$ has the opposite clockwise ordering.
The point of transversal intersection has no significance: 
it is actually two distinct points 
of the curve which
coincided after an immersion into a plane (a.k.a drawing).
Now $\hom_\cC(m,n)$ is the quotient space of
the $\bQ$-span of all $(m,n)$-Jacobi diagram subject to the $AS$-relation and the $IHX$-relation:

\begin{picture}(250,50)(-30,0)
\includegraphics[scale=0.4]{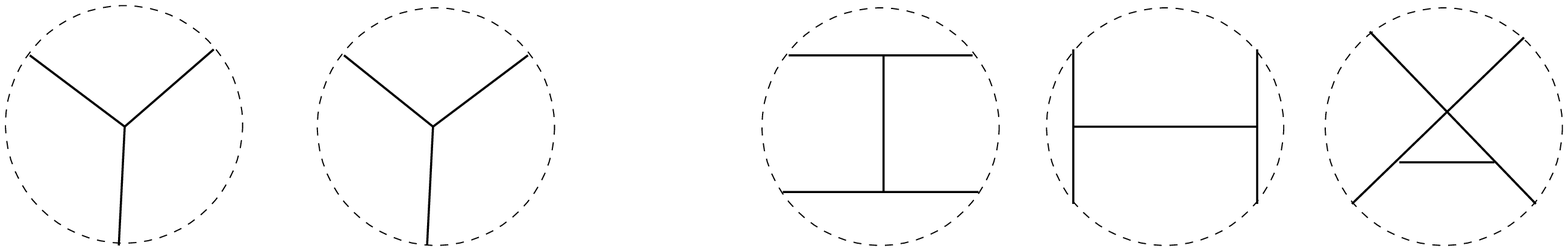}
\put(-246,25){\makebox(0,0){$+$}}
\put(-104,25){\makebox(0,0){$-$}}
\put(-50,25){\makebox(0,0){$+$}}
\put(-211,21){\makebox(0,0){$\star$}}
\put(-182,25){\makebox(0,0){$=0$,}}
\put(10,25){\makebox(0,0){$=0$.}}
\end{picture}

\noindent The dashed circles bound neighbourhoods in Jacobi diagrams $X_1$, $X_2$ for
the first relation and 
 $Y_1$, $Y_2$, $Y_3$ for
the second one, which are identical except for these neighbourhoods.
The relations mean that
$$
X_1+X_2=0 \ \mbox{ and } \ Y_1-Y_2+Y_3=0.
$$

The composition
and the tensor product of morphisms is done by ``stacking'' the boxes:

\begin{picture}(250,60)(-30,0)
\includegraphics[scale=0.4]{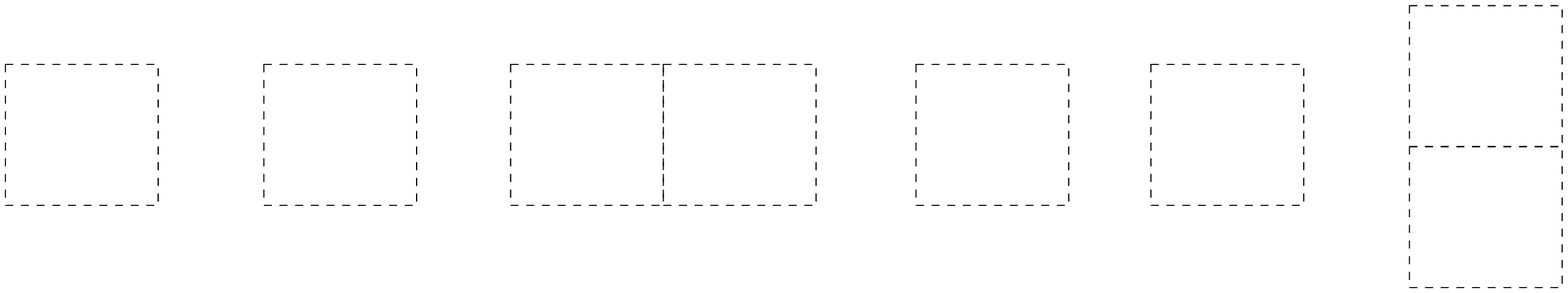}
\put(-237,31){\makebox(0,0){$Y$}}
\put(-161,31){\makebox(0,0){$Y$}}
\put(-64,31){\makebox(0,0){$Y$}}
\put(-15,13){\makebox(0,0){$Y$}}
\put(-287,31){\makebox(0,0){$X$}}
\put(-191,31){\makebox(0,0){$X$}}
\put(-110,31){\makebox(0,0){$X$}}
\put(-15,41){\makebox(0,0){$X$}}
\put(-262,30){\makebox(0,0){$\circ$}}
\put(-88,30){\makebox(0,0){$\otimes$}}
\put(-140,30){\makebox(0,0){;}}
\put(-213,30){\makebox(0,0){$=$}}
\put(-40,30){\makebox(0,0){$=$}}
\end{picture}

\noindent Thus, the tensor product of morphisms 
is just a union while 
the composition is gluing along the corresponding part of the boundary. 
The symmetric group $S_n$ embeds in a way that
$\widetilde{\pi}$ is a union of $n$ intervals connecting $k$ with $\pi (k)^\prime$:

\begin{picture}(200,70)(-150,0)
\includegraphics[scale=0.5]{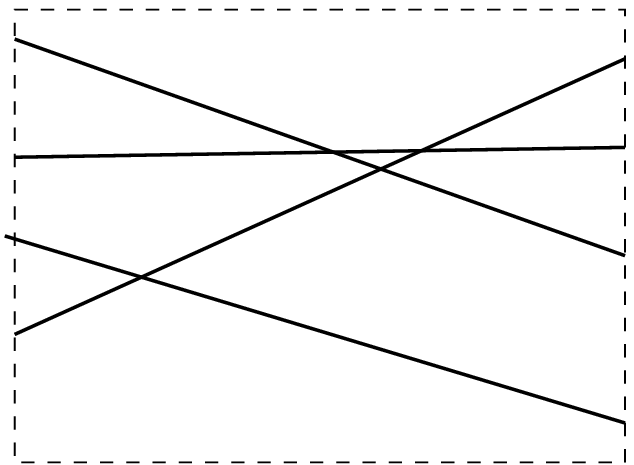}
\put(5,63){\makebox(0,0){$1$}}
\put(5,48){\makebox(0,0){$2$}}
\put(5,20){\makebox(0,0){$\vdots$}}
\put(5,3){\makebox(0,0){$n$}}
\put(17,32){\makebox(0,0){$\pi^{-1}(1)$}}
\put(-95,63){\makebox(0,0){$1^\prime$}}
\put(-103,47){\makebox(0,0){$\pi(2)^\prime$}}
\put(-103,33){\makebox(0,0){$\pi(n)^\prime$}}
\put(-103,16){\makebox(0,0){$\pi(1)^\prime$}}
\put(-70,20){\makebox(0,0){$\vdots$}}
\put(-15,27){\makebox(0,0){$\vdots$}}
\put(-140,34){\makebox(0,0){$\widetilde{\pi}=$}}
\put(44,34){\makebox(0,0){.}}
\end{picture}

The $\cC$-algebra $\fg_M = (1,\mu,\tau,\gamma)$ where

\begin{picture}(250,50)(-40,0)
%
\includegraphics[scale=0.4]{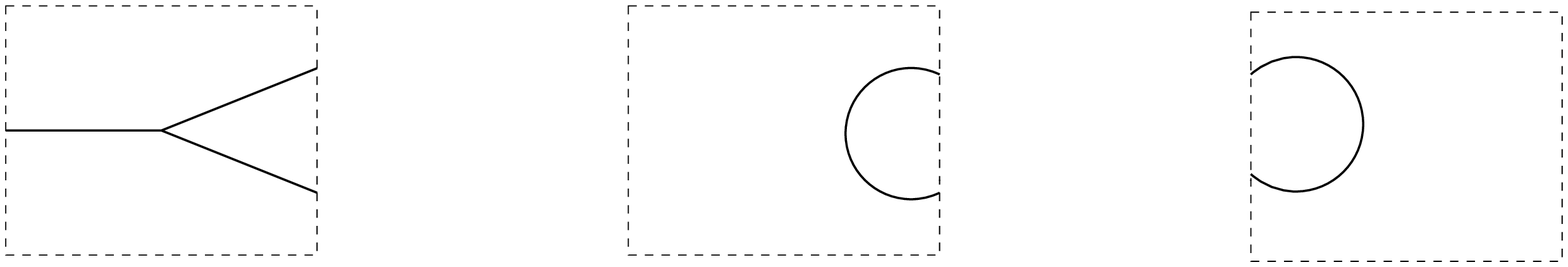}
\put(-290,25){\makebox(0,0){$1^\prime$}}
\put(-63,34){\makebox(0,0){$1^\prime$}}
\put(-63,12){\makebox(0,0){$2^\prime$}}
\put(-110,35){\makebox(0,0){$1$}}
\put(-110,12){\makebox(0,0){$2$}}
\put(-222,34){\makebox(0,0){$1$}}
\put(-222,12){\makebox(0,0){$2$}}
\put(-315,23){\makebox(0,0){$\mu =$}}
\put(-191,23){\makebox(0,0){, $\tau =$}}
\put(-88,23){\makebox(0,0){, $\gamma =$}}
\end{picture}

\noindent is a metric Lie algebra \cite[Prop. 2.4]{NiWi}.
The reader can easily verify this. Jacobi identity and anticommutativity
follow from
the AS-relation and the IHX-relation. The axioms of a metric object,
as well as
invariance of the form and co-form are evident on a drawing.
The universal property of this metric Lie algebra
is more subtle \cite[Prop. 2.5]{NiWi}.
We state it as 
a separate proposition.
\begin{prop}
Given a commutative $\bQ$-algebra $\bK$ and a metric Lie $\bK$-algebra 
$(\cD,\fg,m,t,c)$, there exists a unique specialization $\psi=(\psi_1,\psi_2,\psi_3) : \fg_M \rightarrow \fg$
such that 
\begin{mylist}
\item[(1)] $\psi_1 : \bQ\rightarrow \bK$ is 
the $\bQ$-algebra  structure on $\bK$,
\item[(2)] $\psi_2 = (\psi_{21},\psi_{22},\psi_{23}): \cC \otimes \bK  \rightarrow \cD$ is a tensor functor,
  defined on objects by $\psi_{21} (n) =\fg^{\otimes n}$, and uniquely
  determined on morphisms by the requirements that
\begin{mylist}
\item[(a)] $\psi_{21} (\mu)= m$, 
\item[(b)] $\psi_{21} (\tau ) = t$, 
\item[(c)] $\psi_{21} (\gamma)=c$,
\item[(d)]  $\psi_{22,n,m}: \fg^{\otimes n}\otimes \fg^{\otimes m}
  \rightarrow \fg^{\otimes (n+m)}$ is the composition of associativity
  constraints in $\cD$,
\item[(e)]  $\psi_{23}=I_\1$, the identity morphism of $\1_\cD$,
\end{mylist}
\item[(3)] $\psi_3=I_\fg$, the identity morphism of $\fg$.
\end{mylist}
\end{prop}
The specialization $\psi$ 
is an example of isospecialization.
Its constituent 
is a functor 
$\psi_2: \cC \otimes \bK  \rightarrow \cD$
that turns $\fg$ into an algebra over the prop
$\cC\otimes\bK$.
In the language of props $\cC\otimes\bK$
is the universal metric Lie algebra over $\bK$:
$\cC\otimes\bK$-algebras are the same as metric Lie $\bK$-algebras \cite[Cor. 2.2]{NiWi}.

In the rest of the section we will show
that $\fg_M$ is not quasisimple.
Consider $\hom^k_\cC(n,m)$, the $\bQ$-linear span of all those
diagrams
in $\hom_\cC(n,m)$ that contain precisely $k$ connected components.
The $AS$-relation and $IHX$-relation do not change the number of connected components.
Hence the spaces
$\hom^k_\cC(n,m)$ for different $k$ have zero intersections,
while $\hom_\cC(n,m)$ splits into a direct sum
$\hom_\cC(n,m)=\oplus_{k=0}^\infty \hom^k_\cC(n,m)$.
Notice that $\hom^0_\cC(n,m)\neq 0$ if and only if $n=m=0$: it is
spanned by the empty diagram.
The tensor products preserve this grading, but the compositions,
in general, destroy it. Let us summarize properties of the grading.
\begin{prop} 
\label{Lambda}
The following statements hold.
\begin{mylist}
\item[(1)] 
$
\hom^k_\cC(n,m) \otimes \hom^s_\cC(n,m) \subseteq \hom^{k+s}_\cC(n,m)
$
for all $k$, $s$, $n$, $m$.
\item[(2)]  
$
\hom^k_\cC(0,m) \circ \hom^s_\cC(n,0) \subseteq \hom^{k+s}_\cC(n,m)
$
for all $k$, $s$, $n$, $m$.
\item[(3)] Under this grading 
the scalar ring $\bK_\cC = \hom_\cC(0,0)$ is a graded
algebra, isomorphic to the symmetric algebra of 
$
\hom^1_\cC(0,0)
$.
\item[(4)]  Under this grading 
each $\hom_\cC(n,m)$ is a graded
$\bK_\cC$-module.
\end{mylist}
\end{prop}
\begin{proof}
The tensor product of morphisms is a disjoint union. Hence,
the number of components of the tensor product is the sum, and (1)
follows. A similar argument proves (2).
(4) immediately follows from (1).

It follows from (1) 
that  $\hom_\cC(0,0)$ is a graded  commutative algebra.
Thus, there is a natural homomorphism
$\varphi : S(\hom^1_\cC(0,0))\rightarrow\hom_\cC(0,0)$
of graded algebras.
It is surjective because each diagram is a product of its components.

Let us show that $\varphi$ is injective. 
Consider $x_1,x_2\ldots$, all connected Jacobi diagrams without
boundary before the $AS$-relation and the $IHX$-relation are imposed.
There are natural homomorphisms of graded algebras $\varphi_1$,
$\varphi_2$ making the diagram

\hspace{55mm}
\begin{picture}(150,60)(0,0)
\put(0,60){$ S(\hom^1_\cC(0,0))$}
\put(105,60){$\hom_\cC(0,0)$}
\put(70,63){\vector(1,0){30}}
\put(80,55){$\varphi$}
\put(13,8){$\bQ[x_1,x_2\ldots ]$}
\put(42,19){\vector(0,1){35}}
\put(27,35){$\varphi_1$}
\put(58,19){\vector(3,2){56}}
\put(91,32){$\varphi_2$}
\end{picture}

\noindent
commutative.
The kernel of $\varphi_1$
is the ideal $I_1$ generated by all $y_1+y_2$ from the $AS$-relation
and all $y_1-y_2+y_3$ from the $IHX$-relation, where $y_1$, $y_2$ and
$y_3$
are connected.
Similarly,
The kernel of $\varphi_2$
is the ideal $I_2$ generated by all $z_1+z_2$ from the $AS$-relation
and all $z_1-z_2+z_3$ from the $IHX$-relation, where $z_1$, $z_2$ and
$z_3$
are are no longer necessarily connected.
Both relations operate on just one component,
that is, if  $z_1+z_2$ is an $AS$-relation,
$z_1=x_1z$ and $z_2=x_2z$ where
$x_i$ is the component affected by the $AS$-relation
and $z$ is the union of the remaining components.
Thus, $z_1+z_2= (x_1+x_2)z\in I_1$.
Similarly,  $z_1-z_2+z_3= (x_1-x_2+x_3)z\in I_1$
for all $IHX$-relations. 
It follows that $I_1=I_2$ and $\varphi$ is injective.
\end{proof}

Since connected Jacobi diagrams span $\hom^1_\cC(0,0)$,
one can choose a basis $y_1,y_2\ldots$ among them. It follows that 
$\bK_\cC = \hom_\cC(0,0)$
is isomorphic the polynomial algebra
$\bQ[y_1,y_2\ldots ]$
on the chosen connected $(0,0)$-Jacobi diagrams such as

\begin{picture}(250,50)(-60,0)
\includegraphics[scale=0.4]{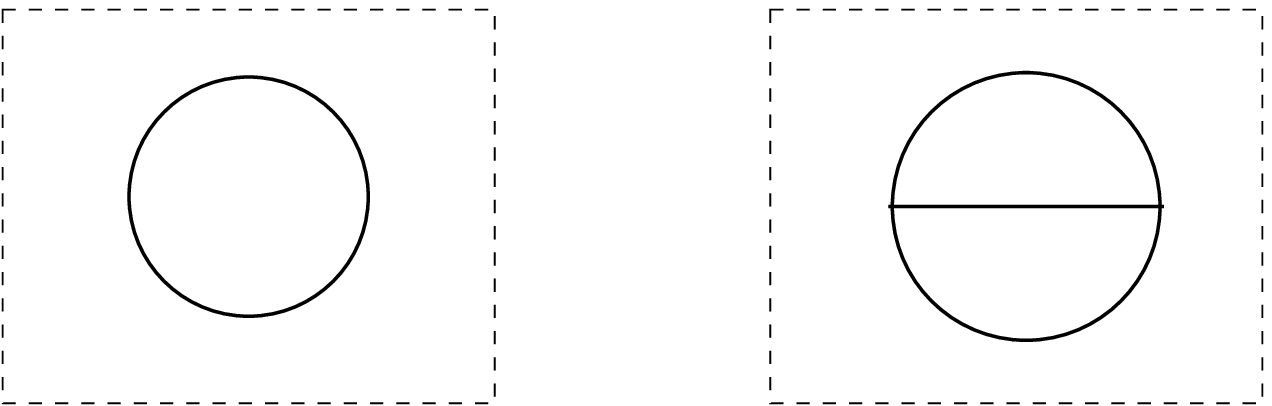}
\put(-160,23){\makebox(0,0){$\delta =$}}
\put(-71,23){\makebox(0,0){, $\theta =$}}
\put(16,23){\makebox(0,0){, $\cdots$}}
\put(60,26){\makebox(0,0){$\in \hom^1_\cC (0,0)$,}}
\end{picture}

\noindent while
$\Endd_{\fg_M} (\fg_M) = \hom_\cC (1,1)$ is a graded
$\bK_\cC$-module. The homomorphism
$\bK_\cC \rightarrow \hom_\cC (1,1)$ 
is the action on the identity element $v\mapsto v\otimes I_1$. 
Since  $I_1\in \hom^1_\cC (1,1)$ and $\hom^0_\cC (0,0)=\bQ$,
any connected non-identity diagram from $\hom^1_\cC (1,1)$
is not in the image of the map 
$\bK_\cC \rightarrow \Endd_{\fg_M}
(\fg_M)$.
In particular, the following element $\phi$
is not in the image:

\begin{picture}(250,50)(-60,0)
\includegraphics[scale=0.4]{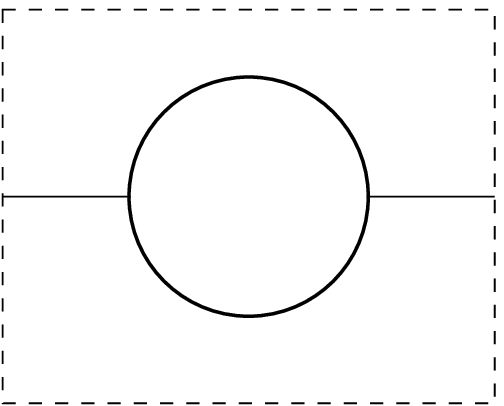}
\put(-71,23){\makebox(0,0){$\phi =$}}
\put(64,23){\makebox(0,0){.}}
\put(34,26){\makebox(0,0){$\in \hom_\cC (1,1)$}}
\end{picture}

\subsection{Vogel's ring $\Lambda$}
\label{VrL}
The symmetric group $S_n$ acts on the vector space $\hom^k_\cC(n,0)$ by permuting the inputs.
As a vector space, Vogel's ring $\Lambda$ is $\hom^1_\cC(3,0)^{S_3,\varepsilon}$,
the skew invariants with respect to the sign character $\varepsilon$ of $S_3$. 
The multiplication is via insertion of one diagram into any trivalent point of the second diagram:

\begin{picture}(250,50)(-40,0)
\includegraphics[scale=0.4]{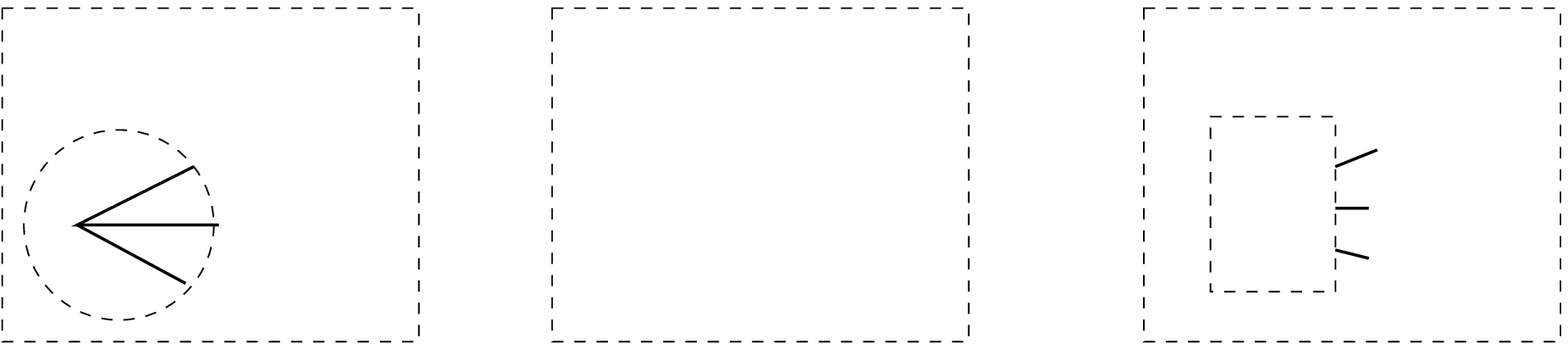}
\put(-40,18){\makebox(0,0){$Y$}}
\put(-110,23){\makebox(0,0){$Y$}}
\put(-10,40){\makebox(0,0){$X$}}
\put(-166,40){\makebox(0,0){$X$}}
\put(-148,23){\makebox(0,0){$\cdot$}}
\put(-69,23){\makebox(0,0){$=$}}
\put(13,23){\makebox(0,0){.}}
\end{picture}

Observe that connectedness and skew-invariance 
of elements of
$\hom^1_\cC(3,0)$
are crucial for this product to be well-defined.
Once a trivalent singular point is chosen, there are six ways to insert 
a diagram from $\hom^1_\cC(3,0)$. 
Skew-invariance removes this dependency.
If two singular points are connected by an edge, 
one can see that insertions into each
of these two points give the same result. 
Consequently, the
connectedness eliminates the dependency of insertion on the choice of a singular point.
Vogel proves these facts \cite[Prop 3.2]{Vog1}.
At the same time he establishes that the insertion
defines a structure of a
$\Lambda$-module on the linear span of connected diagrams with a non-empty set of trivalent
singular points
$\hom^s_\cC(m,n)\subseteq\hom_\cC(m,n)$.

While associativity of $\Lambda$ is obvious, commutativity 
requires a subtle argument \cite[Prop 4.8]{Vog2}.
Notice that Vogel works over any commutative coefficient ring $R$
and proves the identity $12xy = 12yx$ in $\Lambda_R$. 
In our case $1/12 \in R = \bQ$,
therefore the ring $\Lambda$ is commutative.
Here are some elements of $\Lambda$:

\begin{picture}(250,50)(-40,0)
\includegraphics[scale=0.4]{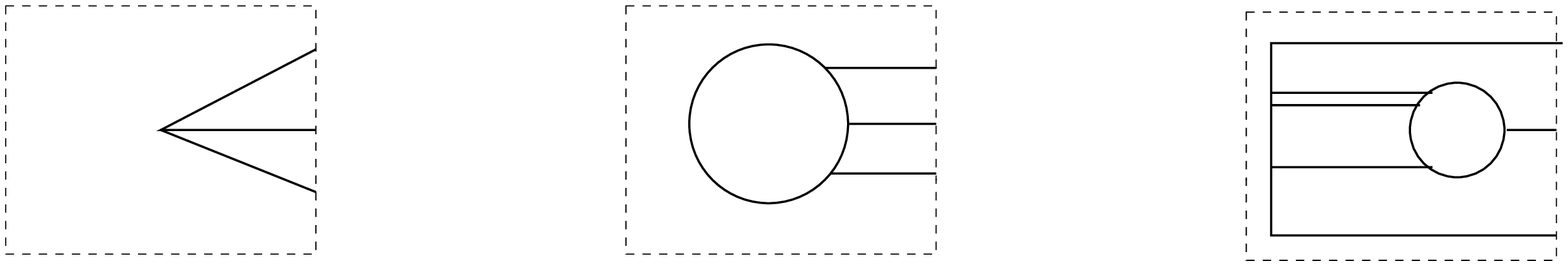}
\put(-43,27){\makebox(0,0){$\vdots$}}
\put(-37,24){\makebox(0,0){$n$}}
\put(-315,23){\makebox(0,0){$1 =$}}
\put(-191,23){\makebox(0,0){, $t =$}}
\put(-88,23){\makebox(0,0){, $x_n =$}}
\put(13,23){\makebox(0,0){.}}
\end{picture}

\noindent Let us note that $x_1=2t$ and $x_2=t^2$ \cite{Vog2}.

The ring $\Lambda$ is naturally graded: $\Lambda_n$ is spanned by the diagrams with $2n+1$ trivalent singular points.
The Poincare series of $\Lambda$ is not known.

The ring $\Lambda$ is closely related to 
a subalgebra 
$\widetilde{\Lambda} = \bQ [\sigma_1] \oplus \omega \bQ[z_1,z_2,z_3]^{S_3}$ 
of $\bQ[z_1,z_2,z_3]^{S_3}$
where
$\omega = (\sigma_1 + z_1)(\sigma_1 + z_2)(\sigma_1 + z_3) = \sigma_3 +\sigma_2 \sigma_1 + 2 \sigma_1^3$
and 
$\sigma_1$, $\sigma_2$, $\sigma_3$ 
are the elementary symmetric polynomials.
Kneissler constructs further elements $\chi_n \in \widetilde{\Lambda}$ by
$$
\chi_0 = 0, \
\chi_1 = 2\sigma_1, \
\chi_2 = \sigma_1^2, \
\chi_{n+3} =
\sigma_1 \chi_{n+2}
- \sigma_2 \chi_{n+1}
+ \sigma_3 \chi_{n}
+ \frac{\sigma_2 \sigma_1^{n+1}}{2}
- \frac{\sigma_3 \sigma_1^{n}}{2}
- \sigma_3 (2 \sigma_1)^{n}
$$
and a homomorphism of graded algebras \cite{Knei}
$$
\varphi : \widetilde{\Lambda} \rightarrow \Lambda, \ \
\varphi (\chi_n) = x_n.
$$
Vogel constructs a polynomial $\pi\in \widetilde{\Lambda}$ such that
$\varphi (\pi)=0$ and conjectures that $\widetilde{\Lambda}/ (\pi)\cong \Lambda$ \cite{Vog1}.

There is a canonical invariant non-degenerate form (the Killing form)
$K: \fg\otimes\fg \rightarrow \bC$
on a simple complex Lie algebra
$\fg$.
It determines a canonical isospecialization
$\psi: \fg_M \rightarrow \fg$
that, in its turn, 
defines a canonical ring homomorphism
$\Theta_\fg : \Lambda \rightarrow \bC$.
It is sensible to call it {\em the Vogel character of Lie algebra}
$\fg$.
Let us clarify its origin
\cite{Vog1,Vog2}.
An element
$X\in\Lambda$ 
gives a $\fg$-invariant map 
$\psi_2 (X): \wedge^3 \fg\rightarrow \bC$.
The space of such maps is one-dimensional, hence,
$\psi_2 (X)=\alpha \psi_2 (1_\Lambda)$ for some complex number 
$\alpha$.
\begin{prop}
In the notation of the last section
the map
$\Theta_\fg (X) = \alpha$
is a ring homomorphism.
\end{prop}
\begin{proof}
It is obvious that 
$\Theta_\fg$ 
is linear. Let
$\widetilde{Y}\in\hom_\cC (3,3)$
be the result of removal of a neighbourhood of a singular point from 
$Y$.
Now one can represent products in Vogel ring via compositions:
$$
XY = X \circ \widetilde{Y}
\ \ \ 
\mbox{ and }
\ \ \ 
Y = 1_\Lambda \circ \widetilde{Y},
$$
which implies that 
$
\psi_2 (XY) = \psi_2 (X)\circ \psi_2 (\widetilde{Y})
= \Theta_\fg (X)  \psi_2 (1_\Lambda)\circ \psi_2 (\widetilde{Y})
= \Theta_\fg (X)  \psi_2 (Y)
$
and, finally, 
$
\Theta_\fg (XY) = 
\Theta_\fg (X) \Theta_\fg (Y)$.
\end{proof}

\subsection{Vogel's universal Lie algebra}
Vogel has realized how to combine 
the character
$\Theta_\fg : \Lambda \rightarrow \cC$
and the isospecialization
$\psi: \fg_M\rightarrow \fg$ 
into a single structure
by forcing the action of $\Lambda$ on  the category $\cC$ \cite{Vog2}.
The new tensor category $\cC^\prime$ is a quotient of $\cC\otimes_{\bQ}\Lambda$
by the Vogel relations:

\begin{picture}(250,50)(-40,0)
\includegraphics[scale=0.4]{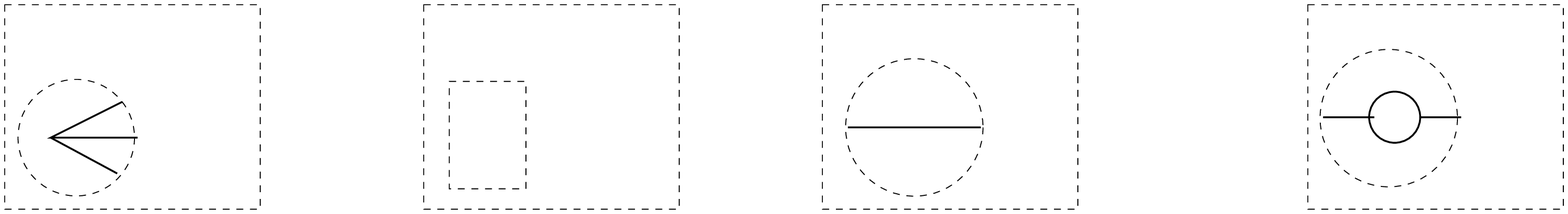}
\put(-302,38){\makebox(0,0){$X$}}
\put(-208,38){\makebox(0,0){$X$}}
\put(-240,17){\makebox(0,0){$v$}}
\put(-118,38){\makebox(0,0){$Y$}}
\put(-10,38){\makebox(0,0){$Y$}}
\put(-272,23){\makebox(0,0){$\otimes v \ \ -$}}
\put(-185,23){\makebox(0,0){$\otimes 1 \; $;}}
\put(-86,23){\makebox(0,0){$\otimes 2t \ \ -$}}
\put(10,23){\makebox(0,0){$\otimes 1$.}}
\end{picture}

\noindent Thus, $\cC^\prime$ is also a prop with 
morphisms
$
\hom_{\cC^\prime} (m,n) := \hom_{\cC} (m,n) \otimes_{\bQ} \Lambda / I_{m,n}
$
where $I_{m,n}$ 
is the $\Lambda$-span of all 
$$
X_1 \otimes v - X_2 \otimes 1 \ , \  \ Y_1 \otimes 2t - Y_2 \otimes 1
$$
where
$v\in\Lambda$, $X_1$ and $X_2$ differ as on the diagrams in the first relation,
$Y_1$ and $Y_2$ differ as on the diagrams in the second relation.

\begin{theorem}
\label{vog_m}
\begin{mylist}
\item[(1)] $\cC^\prime$ is a tensor category over $\Lambda [\delta]$.
\item[(2)] $\fg_V = (1,\mu\otimes 1,\tau\otimes 1,\gamma\otimes 1)$ is
a quasisimple Lie $\cC^\prime$-algebra.
\end{mylist}
\end{theorem}
\begin{proof} Observe that the tensor product and the composition of cosets
$$
(f+I_{n,m})\circ (g+I_{k,n}) = fg+I_{k,m} , \
(f+I_{m,n})\otimes (g+I_{m^\prime,n^\prime}) = f\otimes g +I_{m+m^\prime,n+n^\prime}
$$
are well defined. 
To see this, it suffices to establish that the collection of subspaces
$I_{n,m}$ is a tensor ideal, that is,
$$
f\circ I_{k,n} \subseteq I_{k,m} \supseteq I_{n,m}\circ g , \ \ 
h\otimes I_{m,n} \subseteq I_{m+m^\prime,n+n^\prime} \supseteq
I_{m,n}\otimes h .
$$
By $\Lambda$-linearity, we can assume that 
$f\in \hom_\cC (n,m)$,
$g\in \hom_\cC (k,n)$
and
$h\in \hom_\cC (m^\prime,n^\prime)$.
The argument used in
Proposition~\ref{Lambda} is based on the local nature of the
relations.
It works here as well.
Let 
$X_1 \otimes v - X_2 \otimes 1 \in I_{k,n}$ 
where $X_1$ and $X_2$ are diagrams related by the first Vogel relation.
Then
$f\circ X_1$ and $f\circ X_2$ are also related by the first Vogel
relation,
and
$(f\otimes 1) \circ (X_1 \otimes v - X_2 \otimes 1) \in I_{k,m}$.
The same argument works for the second Vogel relation 
and the tensor products.
Thus,  $\cC^\prime$ is a tensor category.

To compute the scalars, let us recall that
the span of the connected diagrams
with at least one singular point
$\hom^s_{\cC} (n,m)$ is a $\Lambda$-module~\cite[Prop 3.2]{Vog1}. 
Proposition~\ref{Lambda} implies that
$$
\bK_\cC = \hom_{\cC} (0,0)  
\cong S(\hom^1_{\cC} (0,0))
= \bQ [\delta] \otimes_\bQ S (\hom^s_{\cC} (0,0)), 
$$
because
$\delta$ 
is the only connected diagram from
$\hom^1_{\cC} (0,0)$
without singular points. 
The element $\theta$ is a free generator of the $\Lambda$-module
$\hom^s_{\cC} (0,0)$ \cite[Cor 4.6]{Vog1}, hence, 
$$
\bK_\cC \cong \bQ [\delta] \otimes_\bQ S (\Lambda \cdot \theta). 
$$
Let
$I^\prime_{0,0}$
be the
$\Lambda$-span of the first Vogel relations inside
$I_{0,0}$.
Since the first Vogel relation 
asserts $\Lambda$-linearity on $S (\Lambda \cdot
\theta)$,
$$
\hom_{\cC} (0,0) \otimes_{\bQ} \Lambda / I^\prime_{0,0}
\cong 
\bQ [\delta] \otimes_\bQ \mbox{Sym}_\Lambda (\Lambda \cdot \theta)/ (\theta -2t\delta)
\cong \Lambda [\delta].
$$
Evidently, the second Vogel relation has no further effect and
$$
\bK_{\cC^\prime} =  \hom_{\cC} (0,0) \otimes_{\bQ} \Lambda / I_{0,0} = 
\hom_{\cC} (0,0) \otimes_{\bQ} \Lambda / I^\prime_{0,0}
\cong \Lambda [\delta],
$$
proving the first statement.
Since $\fg_M$ is a metric Lie algebra,  $\fg_V$ is a metric Lie algebra. 
Furthermore, 
$$
\Endd_{\fg_M}(\fg_M) = \hom_{\cC} (1,1)  
\cong \bK_\cC \otimes_\bQ (\hom^c_{\cC} (1,1) \oplus (\hom^c_{\cC} (0,1)\otimes_\bQ \hom^c_{\cC} (1,0)))
\cong \bK_\cC  \otimes_\bQ (\bQ I_1 \oplus \Lambda \cdot \phi)
$$
because $\phi$ is a free generator of the $\Lambda$-module $\hom^s_{\cC} (1,1)$ \cite[Cor 4.6]{Vog1}
and $\hom^c_{\cC} (0,1) = 0$, $\hom^c_{\cC} (1,0)=0$
\cite[Prop. 4.3]{Vog1}.
Finally, 
$$
\Endd_{\fg_V}(\fg_V) =  \hom_{\cC^\prime} (1,1)  \cong  \bK_\cC  \otimes_\bQ (\bQ I_1 \oplus \Lambda \cdot \phi)
\otimes_{\bQ} \Lambda / I_{1,1} 
\cong \bK_{\cC^\prime}  \otimes_\bQ (\bQ I_1 \oplus \Lambda \cdot \phi) / (\phi -2tI_1) 
\cong \bK_{\cC^\prime} .
$$
%
\end{proof}

We say that a Lie $\cC$-algebra $\fg$ is {\em $V$-simple} if
it is quasisimple and a specialization
$\fg_M\rightarrow \fg$ extends to a specialization
$\fg_V\rightarrow \fg$.

Clearly, all this definition requires is a ring homomorphism 
$\Theta_\fg :\Lambda [\delta ] \rightarrow \bK_\cC$ 
satisfying the Vogel relations. Existence of such a homomorphism
is not clear to us, in general.
Vogel gives a categorical criterion 
for a quasisimple Lie algebra to be $V$-simple \cite{Vog2}
but it is of limited use.
To prove that
a quasisimple Lie superalgebra $\fg$ is $V$-simple \cite{Vog2}
it is easier
to use the homomorphism $\Theta_\fg$ constructed
at the end of Section~\ref{VrL} (notice that $\Theta_\fg (\delta)= \dim_\cC \fg$ in any category). 
In the final chapter we use the same strategy of constructing a ring homomorphism explicitly while
the Vogel categorical criterion fails.
So far, all known to us quasisimple Lie algebras are $V$-simple, 
hence, the following question is interesting.
\begin{que}
Let $\cC$ be a tensor category. Is a quasisimple Lie $\cC$-algebra 
$V$-simple?
\end{que}

Finally, let us discuss the connection between the Vogel plane  and
the Vogel ring. 
By {\em the Vogel plane} one  commonly understands
the weighted projective plane
$\bP_{1,2,3}^2 = \bP^2 / S_3$ \cite{LaMa,MkSV}.
A point\footnote{Landsberg and Manivel \cite{LaMa}, following Vogel
  \cite{Vog2},
assign a different point  
$((2\delta-\alpha):(2\delta-\beta):(2\delta-\gamma))$.} 
$(\alpha:\beta:\gamma)\in \bP_{1,2,3}^2 (\bC)$, 
where
$\alpha$, $\beta$, $\gamma$ 
are nonzero eigenvalues of Casimir operator
on the symmetric square of the adjoint representation
$S^2(\fg)$,
corresponds to a simple complex Lie algebra
$\fg$.
When the invariant form changes, 
the Casimir operator changes by a scalar, hence, the triple
$(\alpha:\beta:\gamma)$
is also defined only up to a scalar. This construction
has an issue with 
exceptional (according to Deligne) Lie algebras by Deligne
\cite{Deli}.
A Lie algebra of type $A_2$, $D_4$, $E_n$, $F_4$ or $G_2$ 
admits only two such eigenvalues:
$\alpha$ and $\beta$.
Nevertheless, it is sensible to define 
$$
\gamma:=5\delta -\alpha-\beta ,
$$
where
$\delta$ 
is the Casimir eigenvalue on the adjoint representation, because
the identity
$\alpha+\beta+\gamma = 5\delta$
holds in exceptional Lie algebras.
There is a more series issue with
$A_1$, it admits only one eigenvalue
$\alpha$. In this case, one assigns a whole line 
$\beta+\gamma = 5 \delta -\alpha$.

Now let us consider
a $V$-simple Lie algebra 
$\fg$ over a ring $\bK$
with Vogel character 
$\Theta_\fg :\Lambda [\delta ] \rightarrow \bK$.
Homomorphisms of graded rings
$$
\bQ[z_1,z_2,z_3]\hookleftarrow \bQ[z_1,z_2,z_3]^{S_3}\hookleftarrow \widetilde{\Lambda} \xrightarrow{\varphi} \Lambda
$$
give another homomorphism
$\widetilde{\Theta}_\fg:=\Theta_\fg\circ \varphi :\widetilde{\Lambda} \rightarrow \bK$.
Rings
$\bQ[z_1,z_2,z_3]^{S_3}$
and
$\widetilde{\Lambda}$
have a common ring of fractions 
$\bQ[z_1,z_2,z_3]^{S_3}[\omega^{-1}] = \widetilde{\Lambda} [\omega^{-1}]$,
hence, if 
$\widetilde{\Theta}_\fg (\omega)$
is invertible in 
$\bK$, 
then one gets a graded homomorphism
$$
\widehat{\Theta}_\fg :\bQ[z_1,z_2,z_3]^{S_3} \rightarrow \bK, \ \
\widehat{\Theta}_\fg (f) = \frac{\widetilde{\Theta}_\fg
  (f\omega)}{\widetilde{\Theta}_\fg (\omega)}.$$
Examples of algebras with undefined
$\widehat{\Theta}_\fg$
are $\sl_2$  and K3-surfaces in the next chapter.
In both cases
$\widetilde{\Theta}_\fg (\omega)=0$.
Another consequence of invertibility of  
$\widetilde{\Theta}_\fg (\omega)$
is that the
$\bK^\times$-orbit
$[\widehat{\Theta}_\fg]$
turns out to be a point of the weighted projective plane 
$\bP_{1,2,3}^2$
over the ring $\bK$.
Recall that the multiplicative group
$\bK^\times$
acts on the set
$\hom (\bQ[z_1,z_2,z_3]^{S_3}, \bK)$
via the grading of the ring $\bQ[z_1,z_2,z_3]^{S_3}$:
$$
(\alpha \cdot \Theta) (x_n) = \Theta (\alpha^{-n}x_n)
\mbox{ for }
\alpha\in \bK^\times, \ 
x_n\in\bQ[z_1,z_2,z_3]_n^{S_3}, \ 
\Theta \in \hom (\bQ[z_1,z_2,z_3]^{S_3}, \bK),
$$
and the set of $\bK$-points
$\bP_{1,2,3}^2 (\bK)$
is a subset of the set of orbits.
Note that $\widehat{\Theta}_\fg$
depends on the metric Lie algebra
$\fg$.
A different choice of a metric associated is linked by a scalar
$\alpha\in \bK^\times$, 
resulting in the homomorphism
$\alpha\cdot\widehat{\Theta}_\fg$.
Thus, the orbit
$[\widehat{\Theta}_\fg]\in\bP_{1,2,3}^2 (\bK)$
depends only on the Lie algebra itself, but not on its metric structure.

\section{Rozansky-Witten invariants} 

\subsection{The category of vector superbundles}
\label{superbundle}
A holomorphic manifold $X$
admits an associated tensor category, 
crucial for Kapranov's approach to Rozansky-Witten invariants
\cite{Kapr,RoWi}. 
We attach a slightly smaller category than Kapranov. We explain the difference after we explain the construction.

The objects in the category $\SVX$ of vector superbundles
are  locally free coherent sheaves $\cF=\cF_0 \oplus \cF_1$.
The tensor product is usual:
$$
(\cF \otimes \cG)_0 = (\cF_0 \otimes \cG_0)\oplus (\cF_1 \otimes \cG_1), \ 
(\cF \otimes \cG)_1 = (\cF_1 \otimes \cG_0)\oplus (\cF_0\otimes \cG_1). \ 
$$
A locally free coherent sheaf $\cF$ gives to two objects:
even $\cF^+ = \cF\oplus 0$ and odd
$\cF^- = 0\oplus \cF$.
The unit object is the even trivial line bundle $\cO_X^+$. 
The symmetric braiding on $\SVX$ is the usual superbraiding
$$
\tau (v_i \otimes w_j) = (-1)^{ij}  w_j \otimes v_i
$$
where $v_i$, $w_i$ are homogeneous of degree $i$ or $j$, i.e. local sections of $\cF_i$ and $\cG_j$, correspondingly.

The hom-sets are slightly unusual:
$$
\hom_{\SVX} (\cF_0\oplus\cF_1 ,\cG_0 \oplus \cG_1) = \oplus_{i,j,n} \Ext^{(i-j)+2n} (\cF_i,\cG_j).
$$
The composition is the cup-product of extensions. The tensor product of two morphisms $f=(f_n)$
and $g=(g_n)$ is defined by $(f\otimes g)_n = f_n\otimes g_n$.

The category $\SVX$  
is a full subcategory of the 2-periodic derived category
$D^2(X)$
of coherent sheaves on $X$
\cite{Kapu}. 
A superbundle $\cF_0\oplus\cF_1$ corresponds to a complex
$\oplus_{n\in\bZ} (\cF_0 [2n] \oplus\cF_1[2n+1])$
with zero differentials.

In the literature some  ``larger'' categories are sometimes considered.
First, one can get more morphisms in by considering all extensions:
$$
\hom_{\SV_1 (X)}  (\cF_0\oplus\cF_1 ,\cG_0 \oplus \cG_1) = \oplus_{i,j} \Ext^{\ast} (\cF_i,\cG_j) 
\; ,
$$
resulting in a category $\SV_1 (X)$: $\SV(X)$ is a subcategory of 
$\SV_1 (X)$ with the same objects but larger sets of morphisms.
The category $\SV_1 (X)$  
is a full subcategory of the unfolded derived category 
$\widehat{D}(X)$,
whose objects and tensor products are from $D^b(X)$,
while morphisms account
for all extensions:
$$
\hom_{\widehat{D}(X)}  ( \cF_*, \cG_*) = \oplus_{n\in\bZ} \hom_{D^b(X)}   (\cF_*,\cG_*[n]).
$$
For us all these larger categories
are irrelevant since 
all the necessary objects and morphisms for the Kapranov   theorem are
in $\SVX$. 

\subsection{Atiyah classes}
Let $\cT_X$ be the tangent sheaf on $X$,
$\cD_X^{< n}$ the sheaf of differential operators with holomorphic coefficients
of order less than $n$.
We have an exact sequence of $\cO_X-\cO_X$-bimodules
$$
0 \rightarrow
\cO_X \longrightarrow 
\cD_X^{< 2} \longrightarrow 
\cT_X \rightarrow 
0.
$$
Notice that on both $\cO_X$ and $\cT_X$ the right and the left actions
of $\cO_X$ coincide but on $ \cD_X^{< 2}$ they are different.
Given 
a locally free coherent sheaf $\cF$,
we get a new exact sequence by tensoring with it
$$
0 \rightarrow
\cF \longrightarrow 
\cD_X^{< 2} \otimes_{\cO_X}\cF  \longrightarrow 
\cT_X  \otimes_{\cO_X}\cF  \rightarrow 
0.
$$
This is extension is the Atiyah class 
$$
A_\cF \in \Ext^1 (\cT_X  \otimes_{\cO_X}\cF , \cF )
\subseteq
\hom_{\SVX} (\TX \otimes \cF,\cF).
$$
Let us remark that the standard Atiyah class is actually $-A_\cF$ 
\cite{Kapr}.
The following theorem has essentially been proved by Kapranov \cite{Kapr},
but a reader may benefit by looking at later reviews \cite{NiWi, RoWi}.

\begin{theorem}
\begin{mylist}
\item[(1)] $\fg_X = (\SVX, \TX, A_{\cT_X})$ is a Lie algebra.
\item[(2)] Every superbundle $\cF=\cF_0\oplus \cF_1$ is a representation of $\fg_X$ with
the action $A_{\cF}$.
\item[(3)] If $\varsigma$ is a holomorphic symplectic form on $X$ 
then $\fg_X$ is a metric Lie algebra with a form $\varsigma$.
\end{mylist}
\end{theorem}
\begin{proof}
Kapranov's original proof establishes this theorem in the bigger category $\SV_1(X)$ \cite{Kapr}.
We just need to point out that everything in Kapranov's proof happens in either odd or even extensions,
so the theorem actually holds in $\SVX$.
In particular,
observe that $\varsigma$ is skew-symmetric and $\TX$ is odd, hence $\varsigma$ is symmetric in $\SVX$.
\end{proof}

At the moment, the difference between $\SVX$ and $\SV_1 (X)$ looks purely cosmetic.
It becomes crucial when one addresses quasisimplicity.

\subsection{Quasisimplicity}
Let $S^n\cF$ (or $T^n\cF$, or $\Lambda^n\cF$ )
be the $n$-th symmetric (or tensor, or exterior)
power of a locally free coherent sheaf $\cF$.
Let us make a couple of general observations before addressing quasisimplicity of $\fg_X$.
First, observe that $H^{m}(X,\cF)=0$ for all $m>\dim_\bC X$
\cite[Cor 4.39]{Vois}.
Secondly, observe\footnote{This corresponds to 
$T^2 L(\omega_1) \cong L(2\omega_1)\oplus L(\omega_2)\oplus  L(0)$
for the representations of $\sp (n)$.}
that for a holomorphic symplectic manifold $X$
$$
\Endd (\cT_X) \cong T^2 \cT_X \cong S^2 \cT_X \oplus \Lambda^2 \cT_X 
\cong S^2 \cT_X \oplus ( \cL_X \oplus \cO_X)
$$
with the trace map splitting the morphism
$\cO_X\rightarrow  \Lambda^2 \cT_X$,
in particular, on an open set 
$U\subseteq X$
$$
\cL_X (U)=\{ F\in 
\Lambda^2 \cT_X (U) \subseteq\Endd (\cT_X (U)) 
\;\mid\;
\mbox{Tr}(F)=0\}.
$$

%

\begin{theorem}
\label{manifold_simple}
Let $X$ be a holomorphic symplectic manifold. 
Then $\fg_X$  is quasisimple if and only if \newline
$H^{2*}(X,S^2\cT_X \oplus \cL_X)=0$.
\end{theorem}
\begin{proof}
The scalars of the category $\SVX$ are
$$
\bK_{\SVX} = \Endd_{\SVX} (\cO_X^+)
= \Ext^{2*} (\cO_{X},\cO_{X})
= H^{2*} (X, \cO_X).
$$
The naturality of the Atiyah class means that 
every homomorphism in $\SVX$ is a homomorphism
of representations of $\fg_X$. Hence,
$$
\Endd_{\fg_X} (\fg_X)
= \Ext^{2*} (\cT_X,\cT_X)
= H^{2*} (T^2\cT_X )
=  
H^{2*} (X,S^2\cT_X) \oplus H^{2*} (X,\cL_X) \oplus H^{2*} (X, \cO_X) 
$$
and the natural map $\bK_\SVX\rightarrow \hom_{\fg_X} (\fg_X, \fg_X)$
is the identity on the third component.
\end{proof}

The even cohomology is easy to control in the case of K3-surfaces.
This is in sharp contrast to the general behaviour of
the symmetric plurigenus $Q_n(X):=H^0 (X, S^n \cT^*_X)$
that is not a topological invariant of a complex manifold $X$ 
\cite{Bru}.
The crucial argument for the next theorem has been explained to us by A. Kuznetsov.
\begin{theorem}
\label{surface_simple}
If $X$ is a K3-surface then
$\fg_X$
is quasisimple.
\end{theorem}
\begin{proof}
By Theorem~\ref{manifold_simple}
it suffices to prove that
$H^0(X,S^2\cT_X)$
vanishes
because $\cL_X=0$ for a surface.
Suppose 
$H^0(X,S^2\cT_X)\neq 0$.
Then we have two sections of 
$$
\Endd(\cT_X) \cong T^2\cT_X = \Lambda^2 \cT_X \oplus S^2 \cT_X:
$$
the skew symmetric identity map $I\in\Endd(\cT_X)$
and some symmetric section $S\in\Endd(\cT_X)$.
The determinant $\det (I+\lambda S)$ is a global function on $X$
for each $\lambda \in \bC$.
Thus, it is constant and one can choose $\lambda_0\in\bC$ such that
$\det (I+\lambda_0 S)=0$. We conclude that
the rank of $S^\prime := I+\lambda_0 S$
is $1$ at a generic point.

Furthermore, since $I$ is skew-symmetric and $S$ is symmetric,
$S^\prime$ does not vanish, so the rank is $1$ at each point.
Thus, the image of $S^\prime$ is an invertible sheaf $\cL$.
Let us show that $\cL$ must be trivial.

If $X$ is not algebraic then the only line bundle on $X$ is trivial.
If $X$ is algebraic then $\cT_X$ is semistable \cite{Miy}
(cf. \cite[Ch. 7.4]{Huy}) 
and 
$\mu_H (\cL)\leq \mu_H (\cT_X)=0$ for any ample divisor $H$.
The dual of the natural map
$\cT_X\rightarrow \cL$ is an embedding
$\cL^\ast \hookrightarrow \cT^\ast_X\cong \cT_X$,
hence
$-\mu_H (\cL)=\mu_H (\cL^\ast) \leq \mu_H (\cT_X)=0$.
Thus, $\cL\cdot H = \mu_H (\cL) = 0$
for any ample divisor $H$, proving that $\cL$ is trivial. 

Since $\cL$ is trivial, $\cL$ and consequently $\cT_X$ has a nonzero section.  
This contradicts
$H^0(X,\cT_X)= 0$.
\end{proof}

%

\subsection{$V$-simplicity}
The specialization $\RW: \fg_M \rightarrow \fg_X$ is known as 
the Rozansky-Witten invariants (a.k.a weight system)
of a holomorphic symplectic manifold $X$.
It is computed up to various degrees of explicitness for many concrete $X$ \cite{NiWi}.
Let
$\varsigma \in H^0(X,\Lambda^{2}\cT^\ast)$
be the holomorphic symplectic form on
$X$,
$\overline{\varsigma}\in H^0(X,\Omega^{0,2})$
its conjugate form,
and 
$[\overline{\varsigma}] \in H^2(X,\cO_X)$ 
the class corresponding to
$\varsigma$
under the natural isomorphism
$H^0(X,\Omega^{0,2})\cong H^2(X,\cO_X)$.
Using $\varsigma$,
one gets natural embeddings $\Lambda^{n+m}\cT^\ast_X \hookrightarrow \hom (T^n \cT_X,T^m\cT_X)$
so that for $\gamma\in \hom_\cC (\fg_M^{\otimes n}, \fg_M^{\otimes m})$
with $k$ singular trivalent points\footnote{Observe that $k+n+m$ is always even.}
$$
\RW_2 (\gamma) \in 
H^k(X,\Lambda^{n+m}\cT^\ast_X)
\hookrightarrow H^k (X, \hom (T^n \cT_X,T^m\cT_X))
\hookrightarrow \hom_{\SVX} (\fg_X^{\otimes n}, \fg_X^{\otimes m}).
$$
In particular, if
$X$ is a K3-surface, then $RW_2(\gamma)=0$ whenever $\gamma$ has at least three trivalent singular points.
For the diagrams with $RW_2(\gamma)\neq 0$ we have \cite{NiWi}
$$
\RW_2 (\delta) = -2, \ 
\RW_2 (\phi) = -24[\overline{\varsigma}], \  
\RW_2 (\theta) = 48[\overline{\varsigma}] \in H^2(X,\cO_X). \ %
$$

\begin{theorem}
\label{surface_Vsimple}
If $X$ is a K3-surface 
then $\fg_X$ 
is a $V$-simple Lie algebra.
\end{theorem}
\begin{proof}
By Theorem~\ref{surface_simple}  the Lie algebra $\fg_X$ is quasisimple and
$$
\Endd_{\fg_X} (\fg_X) \cong H^{2*} (X, \cO_X) = \bC [z] / (z^2)
\ \mbox{ where } \
z= [\overline{\varsigma}] \in H^2(X,\cO_X).
$$
A graded homomorphism
$$
\Theta_X : \Lambda \rightarrow 
\bC [z] / (z^2), \ \ 
\Theta_X(t) = -24 z
$$
is uniquely determined 
since $\Lambda_0=\bQ$ and 
$\Lambda_1=\bQ t$.
We claim that $\Theta_X$ defines a required specialization.
Indeed, most of the relations defining $\cC^\prime$ hold in $\SVX$ 
for the trivial reason:
both sides are zero.
The relations\footnote{It is also clear by applying the metric 
to $\phi = 2tI$
that $\phi = 2tI$ implies  $\theta = 2t \delta$.} 
$\theta = 2t \delta$ and $\phi = 2tI$ hold as both sides become
$48 z$ and $-24 z$.
\end{proof}

Theorem~\ref{surface_Vsimple} proves Westbury's conjecture for a K3-surface.
Westbury has conjectured it for any compact irreducible holomorphic symplectic manifold $Y$.
We finish by stating this conjecture more carefully as a series of questions.
The quasisimplicity of $\fg_Y$ can be established 
by Theorem~\ref{manifold_simple} and boils down to the following
question.\footnote{
Justin Sawon has reported to us an answer to Question 7: surfaces.
Since $H^2(Y,\Omega^2)=H^2(Y,\cO_Y)\oplus H^2(Y,\cL_Y)$, a manifold $Y$ with a positive answer to Question 7 
must have $h^{2,2}(Y)=1$. If $\dim(Y)>2$, then one can use its three Kahler structures
to construct two linearly independent classes in $H^2(Y,\Omega^2)$, in particular, $h^{2,2}(Y)>1$.}
\begin{que}
Which compact irreducible holomorphic symplectic manifolds $Y$
satisfy $H^{2*}(Y,S^2\cT_Y\oplus \cL_Y)=0$?
\end{que}

If $\dim_\bC Y = n$ 
then the scalars of $\SV (Y)$ are
$
\bK_{\SV (Y)} \cong H^{2*} (Y, \cO_Y) = \bC [z] / (z^n)
$
where
$
z= [\overline{\varsigma}] \in H^2(Y,\cO_Y).
$
Thus, the $V$-simplicity of $\fg_Y$ boils down to some explicit identities
on the Rozansky-Witten invariants.
\begin{que}
Does there exist a graded homomorphism
$\Theta_Y:
\Lambda
\rightarrow 
\bC [z] / (z^n)$
that gives rise to a specialization $\fg_V \rightarrow \fg_Y$? 
\end{que}

It would be interesting if
these homomorphisms are compatible for different $Y$ or, at least, for the Hilbert schemes.
\begin{que}
Does there exists a $V$-simple Lie algebra $\fg_H$ over $\bC [z]$ that specializes
to $\fg_{X^{[n]}}$ for all $X$ and $n$
where $X^{[n]}$ is the Hilbert scheme of $n$ points on a K3-surface $X$,
so that the specialization $\fg_M\rightarrow \fg_{X^{[n]}}$ factorises as
$\fg_M\rightarrow \fg_H\rightarrow \fg_{X^{[n]}}$?
\end{que}

\end{document}